\documentclass[12pt,a4paper, oneside, bold,secthm,seceqn,amsthm,ussrhead,reqno]{article}
\usepackage[utf8]{inputenc}
\usepackage[english]{babel}
\usepackage[symbol]{footmisc}
\usepackage{amssymb,amsmath,amsthm,amsfonts,xcolor,enumerate,hyperref, comment,longtable, cleveref}
\usepackage{verbatim}
\usepackage{times}
\usepackage{cite}
\usepackage{pdflscape}
\usepackage{ulem}
\usepackage[mathcal]{euscript}
\usepackage{tikz}
\usepackage{hyperref}
\usepackage{refcheck}
\usepackage{cancel}
\usepackage{stmaryrd}
 
\usepackage{amsfonts}
\usepackage{amssymb}
\usepackage{times}
\usepackage{xcolor}
\usepackage{tkz-graph}
\usepackage{url}
\usepackage{float}
\usepackage{tasks}
\usepackage{array}

\usepackage{cite}
\usepackage{hyperref}

 \usepackage{fancyhdr} 
\usepackage{amsfonts}
\usepackage{amsmath}
\usepackage{eurosym}
\usepackage{geometry}

\usepackage{caption,booktabs}

\theoremstyle{plain}
\newtheorem{theorem}{Theorem}
\newtheorem{lemma}[theorem]{Lemma}
\newtheorem{proposition}[theorem]{Proposition}
\newtheorem{remark}[theorem]{Remark}
\newtheorem{example}[theorem]{Example}
\newtheorem{definition}[theorem]{Definition}
\newtheorem{corollary}[theorem]{Corollary}

\def\Lie{\mathrm{Lie}}

\usepackage{xcolor}

\usepackage{bbm}
\usepackage{graphicx}

\let\<\langle
\let\>\rangle

\def\Pre{\mathrm{Pre}}
\def\pre{\mathrm{pre}}

\def\NAP{\mathrm{NAP}}
\def\Lie{\mathrm{Lie}}

\def\Lie{\mathrm{Lie}}

\def\pre{\mathrm{pre}}

\usetikzlibrary{arrows}

\usepackage{fouriernc}
\usepackage{refcheck}

\begin{document}
 
 \bigskip

\noindent{\Large
Binary perm algebras and alternative algebras}

 \bigskip

\begin{center}

{\bf
  Abay Kunanbayev\footnote{Institute of Mathematics and Mathematical Modeling, Almaty, Kazakhstan; \    abai-aga@mail.ru} \& 
   Bauyrzhan Sartayev\footnote{Corresponding author: Narxoz University, Almaty, Kazakhstan; \    baurjai@gmail.com}
}
 
\end{center}

\

\noindent {\bf Abstract:}
{\it  
In this paper, we describe the defining identities of a variety of binary perm algebras, which is a subvariety of the variety of alternative algebras. In addition, we construct a basis of the free binary perm algebra and find a complete list of identities which satisfy binary perm algebra under commutator. 
}

 \bigskip 

\noindent {\bf Keywords}:
{\it 
perm algebra, 
alternative algebra, 
polynomial identities, 
binary algebra.}

\bigskip 

 \
 
\noindent {\bf MSC2020}:  
17A30, 17A50, 17D05, 17D10.

	 \bigskip

\ 

\


\section*{Introduction}

For the variety of algebras $\mathcal{A}$, we denote by $\mathcal{A}_2$ the variety of algebras whose every two-generated subalgebra belongs to $\mathcal{A}$. The analogous notion we use for $\mathcal{A}_1$, i.e., every one-generated subalgebra of $\mathcal{A}_1$ belongs to $\mathcal{A}$. In this direction, one of the well-studied classes of algebras is the variety of associative algebras, denoted by $\mathcal{A}s$. In \cite{Albert1}, it was proved that if $\mathcal{A}=\mathcal{A}s$, then $\mathcal{A}s_1$ is the variety of power-associative algebras. In \cite{Artin}, it was shown that the defining identities of the variety $\mathcal{A}s_2$ are
$$(ab)c-a(bc)=-(ac)b+a(cb),$$
$$(ab)c-a(bc)=-(ba)c+b(ac).$$
The identities above are known as the defining identities of the variety of alternative algebras, denoted by $\mathcal{A}lt$. Although there have been many studies on alternative algebras, the basis for a free alternative algebra is still unknown. For the variety of Lie algebras, the defining identities of a variety of binary Lie algebras were given in \cite{Gainov}. The result on the variety of binary Leibniz algebras was obtained in \cite{Leib_2}. For the variety of Zinbiel algebras, the result was given in \cite{Zinb_2}.

Let $\mathcal{A}(X)$ be a free algebra generated by a countable set $X$ of the corresponding variety of algebras $\mathcal{A}$. We denote by $\mathcal{A}^{(-)}$ the class of algebras obtained from $\mathcal{A}$ under the commutator. Coming back to associative algebras, the Poincare-Birkhoff-Witt theorem provides the embedding of any Lie algebra into appropriate associative algebra under commutator, and as a consequence, we have $\mathcal{A}s^{(-)}=\mathcal{L}ie$. In the case of alternative algebras, there is still no complete
answer. There is only known that $\mathcal{A}lt^{(-)}\subset \mathcal{M}alcev$, where $\mathcal{M}alcev$ corresponds to the variety of Malcev algebras. For all the mentioned varieties, we have the following inclusions:

\begin{picture}(30,80)
\put(215,57){$\subset$}
\put(190,44){\normalsize\rotatebox[origin=c]{270}{$\rightarrow$}\tiny{$(-)$}}
\put(215,32){$\subset$}
\put(240,44){\normalsize\rotatebox[origin=c]{270}{$\rightarrow$}\tiny{$(-)$}}
\put(183,31){$\mathcal{L}ie$}
\put(183,57){$\mathcal{A}s$}
\put(233,31){$\mathcal{A}lt^{(-)}$}
\put(268,31){$\subset\mathcal{M}alcev$}
\put(233,57){$\mathcal{A}lt$}
\end{picture}
\vspace*{-\baselineskip}

\noindent We see that the varieties on the left side of the diagram have a complete answer regarding free bases and a description of defining identities. However, for the varieties on the right, we do not have answers to the analogical questions. There have been attempts to give a partial answer to these problems. In \cite{Meta-Mal}, it was constructed a basis of the free metabelian Malcev algebra and free alternative algebra with additional identity $[a,b][c,d]=0$. The embedding of some ideals of Malcev algebras into alternative algebra is given in \cite{Filippov1983}. Some ideals of the free alternative algebra were considered in \cite{Filippov-Nil, Filippov-Triv}.
A base of the free alternative superalgebra on one odd generator was constructed in \cite{Shestakov1}. Also, in \cite{Shestakov2} it was proved that every Malcev superalgebra generated by an odd element is special. 

In this paper, we consider a well-known subvariety of the variety of associative algebras, called perm, and we construct a similar diagram for the considered variety.  For more details and history on perm algebras, see \cite{BoHou}, \cite{Perm 1}, \cite{KS2022}, \cite{Perm 2}. In \cite{MS2022}, it was proved that every metabelian Lie algebra can be embedded into appropriate perm algebra under the commutator. As a consequence, we have $\mathcal{P}erm^{(-)}=\mathcal{M}\mathcal{L}ie$, where $\mathcal{P}erm$ (or shortly, $\mathcal{P}$) and $\mathcal{M}\mathcal{L}ie$ are varieties of perm and metabelian Lie algebras, respectively. Then we have

\begin{picture}(30,80)
\put(215,57){$\subset$}
\put(190,44){\normalsize\rotatebox[origin=c]{270}{$\rightarrow$}\tiny{$(-)$}}
\put(215,32){$\subset$}
\put(240,44){\normalsize\rotatebox[origin=c]{270}{$\rightarrow$}\tiny{$(-)$}}
\put(173,31){$\mathcal{M}\mathcal{L}ie$}
\put(173,57){$\mathcal{P}erm$}
\put(233,31){$\mathcal{P}erm_2^{(-)}$}
\put(233,57){$\mathcal{P}erm_2$}
\end{picture}
\vspace*{-\baselineskip}

The paper is organized as follows:

In section 2, we consider a subvariety of the variety of alternative algebras with an additional identity of degree $3$. We prove that the obtained variety coincides with $\mathcal{P}_2$. Moreover, we construct a basis of algebra $\mathcal{P}_2(X)$.

In section 3, we find a complete list of identities which satisfies $\mathcal{P}_2$ algebra under the commutator. Since from the general point of view $\mathcal{P}_2^{(-)}$ is a subclass of the class of Malcev algebras, we construct a basis of $\mathcal{P}_2^{(-)}(X)$ algebra. The main motivation of this work can be presented in the form of a diagram as follows:

\begin{picture}(30,80)
\put(184,30){$\mathcal{P}erm_2\subset \mathcal{A}lt$}
\put(188,60){$\textrm{$\mathcal{P}erm$}\subset \mathcal{A}s$}
\put(173,0){$\mathcal{P}erm_2^{(-)}\subset 
\mathcal{M}alcev$}
\put(183,17){\tiny{$(-)$} \normalsize\rotatebox[origin=c]{270}{$\rightarrow$}}
\put(239,17){\normalsize\rotatebox[origin=c]{270}{$\rightarrow$}\tiny{$(-)$}}
\put(238,47){\normalsize\rotatebox[origin=c]{270}{$\subset$}}
\put(198,47){\normalsize\rotatebox[origin=c]{270}{$\subset$}}
\put(128,30){$\mathcal{M}\mathcal{L}ie$}
\put(278,30){$\mathcal{L}ie$}
\put(158,15){\normalsize\rotatebox[origin=c]{300}{$\subset$}}
\put(266,15){\normalsize\rotatebox[origin=c]{225}{$\subset$}}
\put(148,47){\tiny{$(-)$} \normalsize\rotatebox[origin=c]{225}{$\rightarrow$}}
\put(258,47){\normalsize\rotatebox[origin=c]{325}{$\rightarrow$}\tiny{$(-)$}}
\end{picture}

$ $

In section 4, we forget the associative identity in perm algebra, resulting in a nonassociative permutative (NAP) algebra. Novikov algebra provides an identity for NAP algebra, where the algebra is composed of right-commutative and left-symmetric identities. For more details and history on Novikov algebras, see \cite{SarKol} and \cite{DAS}. An algebra with only a left-symmetric identity is called a $\pre$-$\Lie$ algebra. Interestingly, we have $$\dim(\NAP(n))=\dim(\Pre\textrm{-}\Lie(n)),$$
where $\NAP(n)$ and $\Pre\textrm{-}\Lie(n)$ are 
$n$-th component of operads defined by varieties of algebras NAP and pre-Lie, respectively.

In \cite{RS+}, it was shown that all identities of the algebras $\NAP(X)$ and $\Pre$-$\Lie(X)$ under anti-commutator follow from only commutative identity. We study $\NAP(X)$ algebra under the commutator. While $\Pre\textrm{-}\Lie^{(-)}(X)$ algebra is a Lie algebra, we obtain a completely different result for $\NAP^{(-)}(X)$.

We consider all algebras over a field $K$ of characteristic $0$.

\section{Identities in free binary perm algebra}

In this section, we describe the defining identities of the variety of binary perm algebras and construct a basis of the free binary perm algebra.

\begin{definition}
An associative algebra is called perm if it satisfies the following additional identity:
$$abc-acb=0.$$
\end{definition}

Let us introduce the main object of this paper
\begin{definition}
Let $\mathcal{V}$ be a variety of algebras which is defined by the following identities:
\begin{equation}\label{right-alt}
(a,b,c)+(a,c,b)=0,
\end{equation}
\begin{equation}\label{left-alt}
(a,b,c)+(b,a,c)=0,
\end{equation}
\begin{equation}\label{binary-perm}
(ab)c+(cb)a=(ac)b+(ca)b,
\end{equation}
where $(a,b,c)$ stands for $(ab)c-a(bc)$.
\end{definition}

We denote by $\mathcal{P}$ and $\mathcal{P}(X)$ the variety of perm algebras and free perm algebra generated by a set $X$, respectively. Note that $\mathcal{P}$ is a subvariety of $\mathcal{V}$, and if $\mathcal{P}_2$ is a variety of binary perm algebras, then by straightforward calculations, we obtain $$\mathcal{P} \subseteq \mathcal{P}_2 \subseteq \mathcal{V}.$$

Let us recall a basis of the $\mathcal{P}(X)$ algebra, which is a set of monomials $\mathcal{A}$ of the following form: $$\mathcal{A}=\{x_{i_1}x_{i_2}x_{i_3}\ldots x_{i_n}|\;i_2\leq i_3\leq\ldots\leq i_n\}.$$

The main purpose is to prove that $\mathcal{P}_2=\mathcal{V}$ which can be proved by constructing a basis of the free algebra $\mathcal{V}(X)$.

\begin{lemma}
The following identities hold in algebra $\mathcal{V}(X)$:
\begin{equation}\label{v1}
(ab)(cd)=-((ca)d)b+(a(cd))b+((cd)a)b,
\end{equation}
\begin{equation}\label{v2}
((ab)c)d=((ac)d)b=((ad)b)c,
\end{equation}
\begin{equation}\label{v3}
a((bc)d)=((ac)b)d,
\end{equation}
\begin{equation}\label{v4}
(((ab)c)d)e=(((ab)c)e)d,
\end{equation}
\begin{equation}\label{v5}
((a(bc))d)e=(((ab)c)d)e.
\end{equation}
\end{lemma}
\begin{proof}
It can be proved using computer algebra as software programs  Albert \cite{Albert}.
\end{proof}

For alphabet $X$, we define sets $\mathcal{B}_1$, $\mathcal{B}_2$, $\mathcal{B}_3$ and $\mathcal{B}_4$ as follows:
$$\mathcal{B}_1=\{x_i\},\mathcal{B}_2=\{x_ix_r\},$$
$$\mathcal{B}_3=\{(x_ix_j)x_k,(x_ix_k)x_j,(x_jx_i)x_k,(x_kx_i)x_j,x_k(x_ix_j), (x_jx_i)x_i,(x_ix_i)x_j,(x_ix_i)x_i\},$$
$$\mathcal{B}_4=\{((x_ix_j)x_k)x_l,((x_ix_j)x_l)x_k,((x_jx_i)x_k)x_l,((x_kx_i)x_j)x_l,((x_lx_i)x_j)x_k,(x_k(x_ix_j))x_l,$$
$$((x_jx_i)x_i)x_i,((x_ix_i)x_i)x_j,((x_jx_i)x_i)x_j,((x_ix_i)x_j)x_j,((x_kx_i)x_i)x_j,((x_jx_i)x_i)x_k,((x_ix_i)x_j)x_k\},$$
where $i\leq j\leq k\leq l$. Also, we set
$$\mathcal{B}_n=\{(\ldots((x_{i_1}x_{i_2})x_{i_3})\ldots)x_{i_n}|\;i_2\leq i_3\leq\ldots\leq i_n\},$$
where $n\geq 5$.

\begin{theorem}\label{basis}
The set $\bigcup_i\mathcal{B}_i$ is a basis of algebra $\mathcal{V}(X)$.
\end{theorem}
\begin{proof}
For degrees up to $4$, the result can be verified by software programs Albert \cite{Albert}. Let us show that starting from degree $5$ the spanning monomials of $\mathcal{V}(X)$ are elements from $\bigcup_i\mathcal{B}_i$, where $i\geq 5$. To do that we first show that any monomial of $\mathcal{V}(X)$ can be written as a linear combination of left-normed monomials. Let us start the induction on the length of the monomials. The base of induction is $n=5$. 
By $(\ref{v3})$, $\mathcal{B}_4$ and $(\ref{v5})$, the monomials $x_{i_1}(m_4)$ can be written as a linear combination of left-normed monomials, where $m_n$ is arbitrary nonassociative monomial of degree $n$. Let us show it explicitly:

\noindent by $\mathcal{B}_4$, we can rewrite any monomial of $m_4$ as a linear combination of 
$$\sum_{j,k} ((x_{j_1}x_{j_2})x_{j_3})x_{j_4}+(x_{k_1}(x_{k_2}x_{k_3}))x_{k_4}).$$
From this follows that any monomial $x_{i_1}(m_4)$ can be written as
$$x_{i_1}(\sum_{j,k} ((x_{j_1}x_{j_2})x_{j_3})x_{j_4}+(x_{k_1}(x_{k_2}x_{k_3}))x_{k_4})=\sum_{j,k} x_{i_1}(((x_{j_1}x_{j_2})x_{j_3})x_{j_4})+x_{i_1}((x_{k_1}(x_{k_2}x_{k_3}))x_{k_4}).$$
By (\ref{v3}), we obtain
\begin{multline*}
\sum_{j,k} x_{i_1}(((x_{j_1}x_{j_2})x_{j_3})x_{j_4})+x_{i_1}((x_{k_1}(x_{k_2}x_{k_3}))x_{k_4})=\\
\sum_{j,k} ((x_{i_1}x_{j_3})(x_{j_1}x_{j_2}))x_{j_4}+((x_{i_1}(x_{k_1}x_{k_2}))x_{k_3})x_{k_4}.
\end{multline*}
By $\mathcal{B}_4$, the monomials of the form $((x_{i_1}x_{j_3})(x_{j_1}x_{j_2}))x_{j_4}$ can be written as a linear combination of $(((x_{p_1}x_{p_2})x_{p_3})x_{p_4})x_{p_5}$ and $((x_{p_1}(x_{p_2}x_{p_3}))x_{p_4})x_{p_5}$. It is remained to use (\ref{v5}) on monomials
$((x_{p_1}(x_{p_2}x_{p_3}))x_{p_4})x_{p_5}$ and $((x_{i_1}(x_{k_1}x_{k_2}))x_{k_3})x_{k_4}$.
In subsequent cases, we will omit explicit calculations, since they can be proved similarly.

To represent the monomials $(x_{i_1}x_{i_2})((x_{i_3}x_{i_4})x_{i_5})$ and $(x_{i_1}x_{i_2})(x_{i_3}(x_{i_4}x_{i_5}))$ as a linear combination of left-normed monomials, we use $(\ref{v1})$, $\mathcal{B}_4$ and $(\ref{v5})$. For monomials $((x_{i_1}x_{i_2})x_{i_3})(x_{i_4}x_{i_5})$ and $(m_4)x_{i_5}$, we use $(\ref{v2})$, $\mathcal{B}_4$ and $(\ref{v5})$. For the last case $(x_{i_1}(x_{i_2}x_{i_3}))(x_{i_4}x_{i_5})$, we apply $(\ref{left-alt})$ and use previous cases. 

Let us show that $(m_t)(m_s)$ can be written as the left-normed monomials, where $s+t>5$. If $s>3$ then we use basis monomials of $\mathcal{B}_4$, $(\ref{v3})$ and inductive hypothesis. If $s\leq 3$ then we consider $3$ cases. If $s=1$ then it immediately follows from the inductive hypothesis. If $s=2$ then it follows from $\mathcal{B}_4$, ($\ref{v5}$), ($\ref{v4}$) and inductive hypothesis. If $s=3$ then $m_3$ is $(x_ix_j)x_k$ or $x_i(x_jx_k)$. If $m_3=(x_ix_j)x_k$ then we use $(\ref{v3})$ and inductive hypothesis. If $m_3=x_i(x_jx_k)$ then we need to consider the following additional cases which are $deg(m_t)=3$ and $deg(m_t)>3$. The case $deg(m_t)=3$ can be proved as the base of induction. The case $deg(m_t)>3$ can be proved
using $\mathcal{B}_4, (\ref{v5})$ and $(\ref{v4})$. All considered cases prove that any nonassociative monomial $m_n$ can be expressed as a linear combination of left-normed monomials, where $n\geq 5$.

By $(\ref{v2})$ and $(\ref{v4})$, we can order the generators 
$x_{i_2},x_{i_3},\ldots,x_{i_n}$ in left normed monomial $$(\ldots((x_{i_1}x_{i_2})x_{i_3})\ldots)x_{i_n},$$ where $n\geq 5$. Before we noted that $\mathcal{P}\subset\mathcal{V}$, so the dimension of $\mathcal{P}(X)$ gives the lower bound of dimension $\mathcal{V}(X)$. In the other words, $\dim(\mathcal{P}(X))\leq\dim(\mathcal{V}(X))$. However, starting from degree $5$ we have $\dim(\mathcal{P}(X))=\dim(\mathcal{V}(X))$ which means that the monomials of $\bigcup_i \mathcal{B}_i$ are linearly independent.
\end{proof}

\begin{corollary}
$$\mathcal{P}_2=\mathcal{V}.$$
\end{corollary}
\begin{proof}
To prove it we show that $\mathcal{P}_2\supseteq\mathcal{V}$, i.e., all identities of $\mathcal{P}_2$ hold in $\mathcal{V}$. By straightforward calculations, it can be shown that any identity of degree 3 or 4, whose two generated forms belong to $\mathcal{P}$, follows from (\ref{right-alt}), (\ref{left-alt}) and (\ref{binary-perm}).
By Theorem \ref{basis}, starting from degree $5$ the dimension and basis of algebras $\mathcal{P}(X)$ and $\mathcal{V}(X)$ are the same. Since $\mathcal{P}\subseteq\mathcal{P}_2\subseteq\mathcal{V}$ and $\dim(\mathcal{P})\leq\dim(\mathcal{P}_2)\leq\dim(\mathcal{V})$, the dimension and basis of algebra $\mathcal{P}_2(X)$ from degree $5$ also is the same with $\mathcal{V}(X)$. So, all identities of $\mathcal{P}_2$ follow from (\ref{right-alt}), (\ref{left-alt}) and (\ref{binary-perm}). Finally, we have $\mathcal{P}_2=\mathcal{V}$.
\end{proof}

Although the multiplication table for two basis monomials of $\mathcal{P}(X)$ looks like an associative-commutative algebra and from degree $5$ the dimensions of $\mathcal{P}(X)$ and $\mathcal{P}_2(X)$ are the same, the multiplication table for $\mathcal{P}_2(X)$ is completely different, i.e., the multiplication of two basis monomials can be the sum of several basis monomials. The explicit form of the multiplication table can be obtained from the proof of Theorem $\ref{basis}$.

\section{Binary perm algebras under commutator}

In this section, we consider a class of algebras $\mathcal{P}_2^{(-)}$ which are obtained from 
$\mathcal{P}_2$ under commutator. Since $\mathcal{P}_2$ is a subvariety of $\mathcal{A}lt$, we have the following inclusions:
$$\mathcal{P}_2^{(-)}\subset\mathcal{A}lt^{(-)}\subset \mathcal{M}alcev.$$
The problem of finding a complete list of identities of $\mathcal{A}lt^{(-)}(X)$ algebra is still an open problem. Additionally, the basis of a free Malcev algebra is unknown. Therefore, we describe a complete list of identities of algebra $\mathcal{P}_2^{(-)}(X)$. Moreover, we construct the basis of algebra $\mathcal{P}_2^{(-)}(X)$.

\begin{proposition}
An anti-commutative algebra $\mathcal{P}_2^{(-)}(X)$ satisfies the following additional identities:
\begin{equation}\label{c1}
[[a,b],[c,d]]=[a,[[d,b],c]]+[d,[[b,c],a]]+[b,[[c,a],d]]+[c,[[a,d],b]],
\end{equation}
\begin{equation}\label{c2}
[[[a,d],c],b]=[[[a,d],b],c]+[[[a,b],d],c]-[[[a,b],c],d],
\end{equation}
\begin{equation}\label{c3}
[[[[a,b],c],d],e]=[[[[a,b],d],c],e]=[[[[a,b],c],e],d],
\end{equation}
\begin{equation}\label{c4}
[[J[a,b,c],d],e]=0.
\end{equation}
\end{proposition}
\begin{proof}
Straightforward calculations give the result.
\end{proof}

\begin{theorem}
All identities of $\mathcal{P}_2^{(-)}(X)$ algebra are consequence of anti-commutative, $(\ref{c1})$, $(\ref{c2})$, $(\ref{c3})$ and $(\ref{c4})$ identities.
\end{theorem}
\begin{proof}
It is easy to check that there are no identities of degree $3$ and only $2$ identities of degree $4$. We denote by $\mathcal{C}(X)$ a free anti-commutative algebra with additional identities $(\ref{c1})$, $(\ref{c2})$, $(\ref{c3})$ and $(\ref{c4})$. We have to show that $\mathcal{C}(X)\cong\mathcal{P}_2^{(-)}(X)$.
To do that we first show that starting from degree $5$ the spanning elements of algebra $\mathcal{C}(X)$ are
$$\mathcal{C}_n=\{[[\ldots[[x_{i_1},x_{i_2}],x_{i_3}]\ldots],x_{i_n}]|\;i_1<i_2\leq i_3\leq\ldots\leq i_n\}.$$
Anti-commutative identity, $(\ref{c1})$ and $(\ref{c2})$ provide that any monomial of $\mathcal{C}(X)$ can be written as a linear combination of left-normed monomials with additional condition $i_1<i_2$. To prove it we start induction on the length of monomials from $\mathcal{C}_n$. The base of induction is monomials up to degree $4$. For the monomial of degree $n$, by inductive hypothesis, we consider only one case
$$[[[A_1],x_{i_1}],[[A_2],x_{i_2}]],$$
where $[A_1],[A_2]$ and $x_{i_1},x_{i_2}$ are left-normed words and generators, respectively. Using (\ref{c1}) and anti-commutative identity, we have
\begin{multline*}
[[[A_1],x_{i_1}],[[A_2],x_{i_2}]]=[A_1,[[x_{i_2},x_{i_1}],A_2]]+[x_{i_2},[[x_{i_1},A_2],A_1]]+
[x_{i_1},[[A_2,A_1],x_{i_2}]]+\\
[A_2,[[A_1,x_{i_2}],x_{i_1}]]=-[[[x_{i_2},x_{i_1}],A_2],A_1]-[[[x_{i_1},A_2],A_1],x_{i_2}]-\\
[[[A_2,A_1],x_{i_2}],x_{i_1}]-
[[[A_1,x_{i_2}],x_{i_1}],A_2].
\end{multline*}
By the inductive hypothesis, the monomials $[[[x_{i_1},A_2],A_1],x_{i_2}]$ and $[[[A_2,A_1],x_{i_2}],x_{i_1}]$ can be written as a linear combination of left-normed monomials. For $[[[A_1,x_{i_2}],x_{i_1}],A_2]$, it is enough use (\ref{c2}) and inductive hypothesis. For monomial $[[[x_{i_2},x_{i_1}],A_2],A_1]$, we repeat the given manipulations, and finally, we obtain a linear combination of left-normed monomials and monomial of the form
$$[[x_{j_1},x_{j_2}],[A_j,x_{j_3}]].$$
By (\ref{c1}) and (\ref{c2}), we obtain
\begin{multline*}
[[x_{j_1},x_{j_2}],[A_j,x_{j_3}]]=[x_{j_1},[[x_{j_3},x_{j_2}],A_j]]+[x_{j_3},[[x_{j_2},A_j],x_{j_1}]]+[x_{j_2},[[A_j,x_{j_1}],x_{j_3}]]+\\
[A_j,[[x_{j_1},x_{j_3}],x_{j_2}]]=
-[[[x_{j_3},x_{j_2}],A_j],x_{j_1}]-[[[x_{j_2},A_j],x_{j_1}],x_{j_3}]-[[[A_j,x_{j_1}],x_{j_3}],x_{j_2}]-\\
[[[x_{j_1},x_{j_3}],A_j],x_{j_2}]-[[[x_{j_1},A_j],x_{j_3}],x_{j_2}]+[[[x_{j_1},A_j],x_{j_2}],x_{j_3}].
\end{multline*}
It is remain to use inductive hypothesis.

The identities $(\ref{c3})$ and $(\ref{c4})$ give opportunity to order $x_{i_2},x_{i_3},\ldots x_{i_n}$. Remain to note that $\mathcal{M}\mathcal{L}ie\subset\mathcal{P}_2^{(-)}\subseteq \mathcal{C}$, where $\mathcal{C}$ is a variety defined by anti-commutative, $(\ref{c1})$, $(\ref{c2})$, $(\ref{c3})$ and $(\ref{c4})$ identities. These inclusions give $\dim(\mathcal{M}\mathcal{L}ie(X))\leq\dim(\mathcal{P}_2^{(-)}(X))\leq\dim(\mathcal{C}(X))$. A basis and dimension of free metabelian Lie algebra was given in \cite{Bahturin1973}, and starting from degree $5$, we have $\dim(\mathcal{M}\mathcal{L}ie(X))=\dim(\mathcal{C}(X))$ which means that the bases of $\mathcal{P}_2^{(-)}(X)$ and $\mathcal{C}(X)$ algebras are the same. We obtain the required isomorphism.
\end{proof}

\begin{theorem}
The identities $(\ref{c3})$, $(\ref{c4})$ are consequence of anti-commutative identity and identities $(\ref{c1})$, $(\ref{c2})$. The anti-commutative identity and identities $(\ref{c1})$, $(\ref{c2})$ are independent.
\end{theorem}
\begin{proof}
It can be proved using computer algebra as software programs  Albert \cite{Albert}.
\end{proof}

\begin{remark}
It is worth noting that $\mathcal{P}_2^{(+)}(X)$, as $\mathcal{A}lt^{(+)}(X)$, is a Jordan algebra. In \cite{MS2022}, it was proved that $\mathcal{P}^{(+)}(X)$ is Jordan algebra with the additional identity of the following form:
\begin{equation}\label{eq22}
\{\{a,b\},\{c,d\}\}=\{\{a,d\},\{b,c\}\}.
\end{equation}
In addition, it is proved that all identities in algebra $\mathcal{P}^{(+)}(X)$ follow from listed identities.
By direct calculations, one can be proved that $\mathcal{P}_2^{(+)}(X)$ satisfies $(\ref{eq22})$. Since $\mathcal{P}^{(+)}\subseteq\mathcal{P}_2^{(+)}$,
we obtain that all identities in algebra $\mathcal{P}_2^{(+)}(X)$ follow from commutative, Jordan and $(\ref{eq22})$ identities.
\end{remark}

\section{NAP algebra under commutator}

A basis of free nonassociative permutative algebra (shortly, NAP) was given in \cite{RSbasis}, which is stated as follows: let $X=\{x_1,x_2,\ldots\}$ be a countable set with order $x_1<x_2<\cdots$, and let $\NAP(X)$ be a free NAP algebra generated by $X$. For given nonassociative monomials $u,v\in\NAP(X)$, we write $u<v$ if one of the following conditions holds:
\begin{itemize}
    \item $\deg(u)<\deg(v)$.
    \item $\deg(u)=\deg(v),\;u=u_1u_2,\; v=v_1v_2$ and $u_1<v_1$.
    \item $\deg(u)=\deg(v),\;u=u_1u_2,\; v=v_1v_2$, $u_1=v_1$ and $u_2<v_2$.
\end{itemize}
Such an order is called deg-lex. Let $\mathcal{B}$ be the set of monomials of the form 
$$(\cdots((x_iw_1)w_2)\cdots)w_n,$$
where $w_1\leq w_2\leq\ldots\leq w_n$. This set is a basis of $\NAP(X)$.

\begin{theorem}\label{mainthm}
All identities of algebra $\NAP^{(-)}(X)$ follow from anti-commutative identity.
\end{theorem}

\begin{proof}
It is sufficient to prove that $Anti\textrm{-}com(X)\cong \NAP^{(-)}(X)$, where $Anti\textrm{-}com(X)$ is a free anti-commutative algebra. Let us recall the basis of algebra $Anti\textrm{-}com(X)$ inductively: 

\noindent $u$ is a basis monomial if
\begin{itemize}
    \item $u=vw$ and $v<w$,
    \item $v$ and $w$ are basis monomials,
\end{itemize}
where the order on monomials of $Anti\textrm{-}com(X)$ algebra is deg-lex. The set of basis monomials of $Anti\textrm{-}com(X)$ we denote by $\mathcal{C}$.

We define a right deg-lex order on monomials of the set $\mathcal{B}$ as follows: this is deg-lex order but we start to compare subwords of monomials from the right side. Let $a\in \mathcal{C}$ from $\NAP^{(-)}(X)$, and we denote by $\overline{a}$ a leading monomial of $a$ relative to the right deg-lex order. If $a\in \mathcal{C}$ from $\NAP^{(-)}(X)$ then we denote by $\widetilde{a}$ a word with parentheses from $\NAP(X)$, where parentheses replace commutators.

\begin{example}
If $a=[x_1,[x_2,x_3]]$, then 
$$a=x_1(x_2x_3)-x_1(x_3x_2)-(x_2x_3)x_1+(x_3x_2)x_1=x_1(x_2x_3)-x_1(x_3x_2)-(x_2x_1)x_3+(x_3x_1)x_2.$$
By right deg-lex order, we have $\overline{a}=x_1(x_2x_3)$, and if we place instead of commutators parentheses, then one obtains $\widetilde{a}=x_1(x_2x_3)$. Analogically, if $b=[[x_1,x_2],[x_3,x_4]]$, then $\widetilde{b}=(x_1x_2)(x_3x_4)$.
\end{example}

To prove the result, we first prove the necessary lemmas: 

\begin{lemma}\label{lemma1}
If $a\in\mathcal{C}$, then $\widetilde{a}\in \mathcal{B}$ and $\overline{a}=\widetilde{a}$.
\end{lemma}
\begin{proof}
Firstly, let us show that $\widetilde{a}\in \mathcal{B}$. 
We use induction on the length of monomials from $\mathcal{C}$. The base of induction is $n=3$. We have 
$$\widetilde{[x_i,[x_j,x_k]]}=x_i(x_jx_k)\in\mathcal{B},$$
where $j<k$.
If $a=[u,v]$ then by inductive hypothesis $\widetilde{u},\widetilde{v}\in\mathcal{B}$, and 
since $u<v$, we obtain $uv\in\mathcal{B}$.

Let us prove that $\overline{a}=\widetilde{a}$. To prove it we show that 
$$\overline{a}=\overline{[u,v]}=\overline{uv-vu}=\overline{uv}=\widetilde{a}.$$
We use analogical induction on the length of monomials from $\mathcal{C}$. The base of induction is given above. If $a=[u,v]$, then by inductive hypothesis $\overline{u}=\widetilde{u}$, $\overline{v}=\widetilde{v}$, and since $u<v$, we obtain
$$\overline{uv}=\overline{u}\;\overline{v}=\widetilde{u}\;\widetilde{v}=\widetilde{a}.$$
For $vu$, we consider 2 cases:

1. $vu=([v_1,v_2])u$ and $v_2<u$. For the first case, we have 
$$\overline{vu}=(\overline{v_1}\;\overline{v_2})\;\overline{u}=(\widetilde{v_1}\;\widetilde{v_2})\;\widetilde{u}.$$

2. $vu=([v_1,v_2])u$ and $u<v_2$. For the second case, we have
$$\overline{vu}=(\overline{v_1}\;\overline{u})\;\overline{v_2}=(\widetilde{v_1}\;\widetilde{u})\;\widetilde{v_2}.$$

Since $\overline{uv}=\widetilde{u}\;\widetilde{v}$, $\overline{vu}=(\widetilde{v_1}\;\widetilde{v_2})\;\widetilde{u}$ or $(\widetilde{v_1}\;\widetilde{u})\;\widetilde{v_2}$, we obtain
$$\overline{vu}<\overline{uv}=\widetilde{a}.$$

\end{proof}

\begin{lemma}\label{lemma2}
If $a_1,a_2,\ldots,a_n\in\mathcal{C}$, then $\overline{a_i}\neq \overline{a_j}$ for $i\neq j$.
\end{lemma}
\begin{proof}
By Lemma \ref{lemma1}, we have $\overline{a_i}=\widetilde{a_i}$, $\overline{a_j}=\widetilde{a_j}$ and $\widetilde{a_i},\widetilde{a_j}\in\mathcal{B}$. Since the set $\mathcal{B}$ is a basis of $\NAP\<X\>$, we obtain $\widetilde{a_i}\neq\widetilde{a_j}$.
\end{proof}

Now, we are ready to finish proving our result. By Lemma \ref{lemma2}, all monomials from the set $\mathcal{C}$ in $\NAP^{(-)}(X)$ are linearly independent. Since the algebras $Anti\textrm{-}com(X)$ and $\NAP^{(-)}(X)$ have the same bases, we obtain the required isomorphism.
\end{proof}

\begin{corollary}
All identities of $\NAP^{(+)}(X)$ algebra follow from commutative identity.
\end{corollary}
\begin{proof}
This result can be proved similarly to Theorem \ref{mainthm}.
\end{proof}

\subsection*{Acknowledgments}
The authors were supported by the Science Committee of the Ministry of Education and Science of the Republic of Kazakhstan (Grant No. AP???).

\end{document}